\documentclass[11pt,fleqn]{amsart}
\usepackage{amsmath, amsthm, amssymb, amsfonts, verbatim,color}
\usepackage[pdftex]{graphicx}
\usepackage{csquotes}
\usepackage[bookmarks]{hyperref}
\reversemarginpar            

\definecolor{darkgreen}{rgb}{0,0.55,0}

\newtheorem{theorem}{Theorem}[section]

\newtheorem{lemma}[theorem]{Lemma}
\newtheorem{proposition}[theorem]{Proposition}
\theoremstyle{definition}

\theoremstyle{remark}
\newtheorem{remark}[theorem]{Remark}

\numberwithin{equation}{section}
\numberwithin{theorem}{section}

\newcommand{\abs}[1]{\left\vert#1\right\vert}
\newcommand{\R}{{\mathbb{R}}}

\DeclareMathOperator{\dist}{dist}

\DeclareMathOperator*{\esslim}{ess\,lim}

\def\be{\begin{equation}}
\def\ee{\end{equation}}

\def\({\left(}
\def\){\right)}

\def\R{\mathbb{R}}

\def\e{\varepsilon}

 \usepackage{amsaddr}

\author[]{Andres A. Contreras Hip}
\address[A.~Contreras Hip]{Department of Mathematics, The University of Chicago, Chicago, Illinois, USA}
\email{acontreraship@uchicago.edu} 

\author[]{Xavier Lamy}
\address[X.~Lamy]{Institut de Math\'ematiques de Toulouse; UMR 5219, Universit\'e de Toulouse; CNRS, UPS IMT, F-31062 Toulouse Cedex 9, France}
\email{xlamy@math.univ-toulouse.fr}

\author[]{Elio Marconi*}
\address[E.~Marconi]{EPFL B, Station 8, CH-1015 Lausanne, Switzerland.}
\email{elio.marconi@epfl.ch}

\thanks{*Corresponding author}

\begin{document}
\title[Generalized Characteristics for Finite Entropy Solutions]{Generalized Characteristics for Finite Entropy Solutions of Burgers' Equation}


%

\maketitle
{
\rightskip .85 cm
\leftskip .85 cm
\parindent 0 pt
\begin{footnotesize}
{\sc Abstract.} We prove the existence of generalized characteristics for weak, not necessarily entropic, solutions of Burgers' equation
\[
\partial_t u +\partial_x \frac{u^2}{2} =0,
\]
whose entropy productions are signed measures. 
Such solutions arise in connection with large deviation principles for the hydrodynamic limit of interacting particle systems. The present work allows to remove a technical trace assumption in a recent result by the two first authors about the $L^2$ stability of entropic shocks among such non-entropic solutions. 
The proof relies on the Lagrangian representation of a solution's hypograph, recently constructed by the third author. In particular, we prove a decomposition formula for the entropy flux across a given hypersurface, which is valid for general multidimensional scalar conservation laws.



\end{footnotesize}
}

\section{Introduction}
We consider bounded weak (not necessarily entropy) solutions of Burgers' equation
\begin{align*}
\partial_t u +\partial_x \frac{u^2}{2} =0,
\end{align*}
or more generally a scalar conservation law
\begin{align}\label{eq:burgers}
\partial_t u +\partial_x f(u) =0,\quad t > 0,\; x\in\R,
\end{align}
with strictly convex flux $f''\geq\alpha > 0$. 
For any entropy-flux pair $(\eta,q)$ i.e. $\eta''\geq 0$ and $q'=\eta'f'$, the corresponding entropy production of a bounded weak solution $u$ is the distribution
\begin{align}\label{eq:mueta}
\mu_\eta =\partial_t\eta(u) +\partial_x q(u).
\end{align}
Entropy solutions are weak solutions whose entropy production is nonpositive, i.e. $\mu_\eta\leq 0$ for all convex entropies $\eta$, and given any bounded initial condition $u_0(x)$ there exists a unique entropy solution \cite{kruzkov70}.

Here in contrast we consider weak solutions whose entropy productions do not necessarily have a sign: we call \emph{finite-entropy solution} any bounded weak solutions of \eqref{eq:burgers} such that
\begin{align}\label{eq:finiteentropy}
\mu_\eta \text{ is a Radon measure for all convex }\eta,
\end{align}
where $\mu_\eta$ is the entropy production defined in \eqref{eq:mueta}.
This larger class of solutions is relevant in 
 the study of large deviation principles for the hydrodynamic limit of asymmetric interacting particle systems \cite{kipnis-landim,varadhan04}, or for  scalar conservation laws with appropriately small random forcing \cite{mariani10,bellettini-etal}. Known tools fail at completing the large deviation analysis because finite-entropy solutions do not in general have bounded variation ($BV$). Similar issues arise in the study of the so-called Aviles-Giga energy (see \cite[Introduction]{lamy-otto} for more details). In the past years, several works have proved $BV$-like structural properties for finite-entropy solutions \cite{lecumberry,delellis-otto-westdickenberg,marconi-structure,marconi-rectif} but the large deviation principle still seems out of reach. 
 
A key progress would be to obtain good estimates on the distance to entropy solutions in terms of the positive part of the entropy production. Inspired by \cite{krupa-vasseur}, the two first authors recently proposed a {\it relative entropy} approach to that question \cite{stab_shock}, but they had to assume the existence of \textit{generalized characteristics}: Lipschitz curves $x(t)$ such that $x'(t)=f'(u^\pm(t,x(t)))$ for a.e. $t$ at which $u^+(t,x(t))=u^-(t,x(t))$, where $u^{\pm}(t,x(t))$ are the left and right traces of $u$ along $(t,x(t))$. It is well-known that $BV$ solutions admit generalized characteristics starting at any value $x(0)=x_0$ \cite[\S~10.2]{dafermos}, but the argument uses a stronger notion of traces than the one available for finite-entropy solutions (see the discussion in the introduction of \cite{stab_shock}).
Existence of generalized characteristics is also crucial in several recent works using relative entropy methods for hyperbolic systems of conservation laws (see e.g. \cite{krupa--criteria,chen-krupa-vasseur,krupa-vasseur20}).

Our main result Theorem~\ref{main} establishes the existence of generalized characteristics for finite-entropy solutions. As a corollary, the results in \cite{stab_shock} are valid for any finite-entropy solution of Burgers' equation.

\begin{theorem} \label{main}
Let $u\colon [0,T]\times\R\to\R$ be a finite-entropy weak solution of \eqref{eq:burgers} with strictly convex flux $f$. 
For any $x_0\in \R$, there exists a generalized characteristic of $u$ starting at $x_0$, that is, a Lipschitz curve $x\colon [t_0,T]\to\R$ such that $x(0)=x_0$ and
\begin{align*}
x'(t) = f'(u^\pm(t,x(t)))\qquad\text{for a.e. }t\in [0,T]\text{ s.t. }u^+(t,x(t))=u^-(t,x(t)),
\end{align*}
where $u^\pm(t,x(t))$ denote the left and right traces of $u$ along $(t,x(t))$.
\end{theorem}

\begin{remark}
Finite-entropy solutions have traces which are reached strongly in $L^1$:
\begin{align*}
\esslim_{y\to 0^+}\int_0^T \abs{u(t,x(t)\pm y)- u^\pm(t,x(t))}\, dt =0.
\end{align*}
This is proved in \cite{vasseur--traces} for entropy solutions, but the proof uses only a kinetic formulation which is also valid for finite-entropy solutions \cite{delellis-otto-westdickenberg}.
In particular, for a.e. $ t\in [0,T]$ such that  $u^+(t,x(t))\ne u^-(t,x(t))$ we have the Rankine-Hugoniot condition
\begin{equation*}
x'(t) = \frac{f(u^+(t,x(t))) - f(u^-(t,x(t)))}{u^+(t,x(t)) - u^-(t,x(t))}.
\end{equation*}

\end{remark}

The strategy of proof relies  on the Lagrangian representation recently introduced by the third author \cite{marconi-structure,marconi-rectif}.
A central ingredient is valid for general multidimensional scalar conservation laws and is of independent interest: in Theorem~\ref{fluxformula} we obtain a formula for the entropy flux across a hypersurface in terms of the Lagrangian representation. In order to state this result, we consider  finite-entropy solutions of
\begin{align}\label{eq:scl}
\partial_t u +\nabla_x\cdot f(u)=0,\quad t\in (0,T), \; x\in\R^d,
\end{align}
where the flux $f$ is now any $C^2$ function $f\colon\R\to \R^d$, and the finite-entropy condition amounts to
\begin{align*}
\mu_\eta :=\partial_t\eta(u) +\nabla_x\cdot q(u)\quad\text{ is a Radon measure for all convex }\eta,
\end{align*}
and associated entropy flux $q$ given by $q'=\eta'f'$. 
In \cite{marconi-structure} the third author proves the existence of a Lagrangian representation of the hypograph of $u$, that is, a nonnegative finite measure $\omega_h$ on the set of curves
\begin{align*}
\Gamma=\left\lbrace \gamma=(\gamma_x,\gamma_v)\in BV([0,T);\R^d\times [0,1]\colon \gamma_x\text{ is Lipschitz}\right\rbrace,
\end{align*}
with the following three properties.
\begin{itemize}
\item For all $t\in [0,T)$, the pushforward $(e_t)\sharp\omega_h$ by the evaluation map $e_t\colon \gamma\mapsto \gamma(t^+)$ is uniform on the hypograph of $u(t)$,
\begin{align}\label{eq:etomega}
(e_t)\sharp\omega_h=\mathcal L^{d+1}\lfloor \left\lbrace (x,v)\in \R^d\times [0,1] \colon v<u(t,x)\right\rbrace. 
\end{align}
\item The measure $\omega_h$ is concentrated on curves $\gamma\in \Gamma$ satisfying the characteristic differential equation
\begin{align}\label{eq:characlag}
 \gamma_x' (t)=f'(\gamma_v(t))\quad\text{for a.e. }t\in [0,T).
\end{align}
In particular $\gamma_x$ is $S$-Lipschitz for $\omega_h$-a.e. $\gamma$, where $S=\sup |f'|([0,1])$ is the maximal speed.
\item The total variation of $\gamma_v$ is controlled by
\begin{align}\label{eq:totvarlag}
\int_\Gamma \mathrm{TotVar}\:\gamma_v\, d\omega_h <\infty.
\end{align}
\end{itemize}
\begin{remark}\label{r:epirep}
One can also define a Lagrangian representation $\omega_e$ of the epigraph of $u$. This representation satisfies the same properties as the representation of the hypograph, where \eqref{eq:etomega} is replaced by
\begin{align*}
(e_t)\sharp\omega_e=\mathcal L^{d+1}\lfloor \left\lbrace (x,v)\in \R^d\times [0,1] \colon v> u(t,x)\right\rbrace. 
\end{align*}
We will sometimes loosely refer to typical curves chosen according to the measure $\omega_h$
 (respectively $\omega_e$), as curves of the hypograph (respectively epigraph).
\end{remark}

Our second main result is a decomposition formula for the entropy flux across a given hypersurface, along the curves of the Lagrangian representation.

\begin{theorem} \label{fluxformula}
Let $u$ be a finite-entropy solution of \eqref{eq:scl} with Lagrangian representation $\omega_h$ of its hypograph. Let $\Sigma \subset (0,T)\times \R^d$ be a Lipschitz hypersurface. 
Then $\omega_h$-almost every curve $\gamma\in \Gamma$ intersects $\Sigma$ at most a finite number of times, in the sense that
\begin{align*}
\left\lbrace t\in (0,T)\colon (t,\gamma_x(t))\in\Sigma \right\rbrace\quad\text{ is finite.}
\end{align*}
For any open set $U\subset (0,T)\times\R^d$ such that $U\setminus\Sigma$ has two connected components $\Sigma^\pm$, denote by $\nu$ the unit normal vector to $\Sigma$ pointing from $\Sigma^-$ to $\Sigma^+$ and by $u^\pm$ the traces of $u$ on the corresponding sides. Then, for any entropy-entropy flux pair $(\eta,q)$ with 
\begin{align*}
\eta(0)=0,\quad q(0)=0,
\end{align*}
 the entropy flux across $\Sigma$ from $\Sigma^-$ satisfies
\begin{align*}
&\int_\Sigma \nu\cdot (\eta(u^-) , q(u^-))\, \Phi \, d\mathcal H^d 
=
\int_\Gamma \langle F^-_\gamma ,\eta \otimes \Phi \rangle\, d \omega_h (\gamma)\quad\text{for all }\Phi\in C_c^\infty(U),
\end{align*}
where
\begin{align}\label{eq:Fgamma}
\langle F^-_\gamma ,\eta \otimes \Phi \rangle & =\sum_{t\in  I^+_\gamma   }   \eta'(\gamma_v(t^-))\Phi(t,\gamma_x(t)) - \sum_{t\in I_\gamma^- }   \eta'(\gamma_v(t^+))\Phi(t,\gamma_x(t))\\
&\quad
+\sum_{t\in B_\gamma^-}(\eta'(\gamma_v(t^-)) -  \eta'(\gamma_v(t^+)))\Phi(t,\gamma_x(t))\nonumber
\end{align}
Here $I_\gamma^\pm$ and $B_\gamma^-$ are disjoint subsets of the intersection times of $(t,\gamma_x(t))$ with $U\cap\Sigma$, $I_\gamma^+$ corresponding to crossings from $\Sigma^-$ to $\Sigma^+$, $I_\gamma^-$ to crossings from $\Sigma^+$ to $\Sigma^-$, and $B_\gamma^-$ to bounces on $\Sigma$ from $\Sigma_-$.
\end{theorem}

\begin{remark}
Specifically, the sets $I^\pm_\gamma$, $B_\gamma^-$ appearing in Theorem~\ref{fluxformula} are given  by
\begin{align*}
I_\gamma^+ =\left\lbrace t\in (0,T)\colon (t,\gamma_x(t))\in U\cap\Sigma,\, (t\pm\delta,\gamma_x(t\pm\delta))\in \Sigma^\pm\text{ for }0<\delta\ll 1\right\rbrace,\\
I_\gamma^- =\left\lbrace t\in (0,T)\colon (t,\gamma_x(t))\in U\cap\Sigma,\, (t\pm\delta,\gamma_x(t\pm\delta))\in \Sigma^\mp\text{ for }0<\delta\ll 1\right\rbrace,\\
B_\gamma^- =\left\lbrace t\in (0,T)\colon (t,\gamma_x(t))\in U\cap\Sigma,\, (t\pm\delta,\gamma_x(t\pm\delta))\in \Sigma^-\text{ for }0<\delta\ll 1\right\rbrace.
\end{align*}
\end{remark}

\begin{remark}\label{r:epiflux}
For the Lagrangian representation $\omega_e$ of the epigraph of $u$ (see Remark~\ref{r:epirep}), the identity of Theorem~\ref{fluxformula} becomes, 
\begin{align*}
&\int_\Sigma (\eta(u^-) , q(u^-))\cdot \nu\,\Phi(t,x)\, d\mathcal H^d(t,x)
= - \int_\Gamma \langle F^-_\gamma ,\eta \otimes \Phi \rangle\, d \omega_e (\gamma),
\end{align*}
provided $\eta(1)=0$ and $q(1)=0$. 
\end{remark}

Once the first assertion of Theorem~\ref{fluxformula} is established, that typical curves of the hypograph have finite intersection with $\Sigma$, the flux formula \eqref{eq:Fgamma} follows from rather natural manipulation, using the Lagrangian property \eqref{eq:etomega} that $(e_t)\sharp\omega_h=\mathbf 1_{v<u(t,x)}\, dxdv$ in order to link values of $u$ with the Lagrangian representation. The finite intersection property is a consequence of a transverse intersection property: tangential intersections are negligible, essentially thanks to the fact that $\Sigma$ is of codimension 1 while $(e_t)\sharp \omega_h$ is absolutely continuous with respect to the Lebesgue measure.

The proof of Theorem~\ref{main} uses the flux formula of Theorem~\ref{fluxformula} and the property, established in \cite{marconi-rectif}, that for Burgers' equation \eqref{eq:burgers} curves of the hypograph cannot cross from the left curves of the epigraph. This enables us to construct $x(t)$ as a curve that cannot be crossed from the left by any curve of the hypograph, nor from the right by curves of the epigraph. This implies, via the flux formula \eqref{eq:Fgamma} from which some terms can then be dropped, inequalities on the entropy flux across $x(t)$ that can only be satisfied by a generalized characteristic.

The article is organized as follows. In  Section~\ref{s:proof} we prove  Theorem~\ref{main} as a consequence of Theorem~\ref{fluxformula}, whose proof is given in Section~\ref{s:flux}.

\section{Proof of Theorem \ref{main}}\label{s:proof}

In this section we prove Theorem~\ref{main}, that is, the existence of characteristics for finite-entropy solutions of Burgers. Aside from the flux formula from Theorem~\ref{fluxformula}, the main tool is the existence of a curve $x(t)$ that cannot be crossed from the left by (typical) curves of the hypograph, nor from the right by curves of the epigraph. This is the only place where we need the strict convexity of the flux $f$.

\begin{lemma}\label{l:nocross}
Let $u\colon [0,T]\times\R\to [0,1]$ be a finite-entropy  solution of \eqref{eq:burgers} with strictly convex flux $f$. 
For any $x_0\in \R$, there exists a  Lipschitz curve $x\colon [0,T]\to\R$ such that $x(0)=x_0$, and 
\begin{equation} \label{left}
\omega_h (\{ \gamma \in \Gamma \colon \exists t_1 < t_2 ,\; \gamma_x(t_1) < x(t_1) , \gamma_x(t_2) > x(t_2) \}) = 0,
\end{equation}
and
\begin{equation} \label{right}
\omega_e (\{ \gamma \in \Gamma \colon \exists t_1 < t_2,\;  \gamma_x(t_1) > x(t_1) , \gamma_x(t_2) < x(t_2) \}) = 0.
\end{equation}
\end{lemma}

\begin{proof}[Proof of Lemma~\ref{l:nocross}]
For any $(t,x)\in [0,T)\times \R$ we define a curve $\hat\gamma_{t,x}\colon [t,T]\to\R$ by
\begin{align*}
\hat{\gamma}_{t , x} (s) = \inf \{ y \in \R : \omega_h \{ \gamma \in \Gamma : \gamma_x (t) < x , \gamma_x (s) > y \} = 0 \}.
\end{align*}
In other words, $\hat\gamma_{t,x}(s)$ is the right-most value of $y$ that can be attained at time $s\geq t$ by curves of the hypograph passing left from $x$ at time $t$. Note that $\hat\gamma_{t,x}$ is $S$-Lipschitz because $\gamma_x$ is $S$-Lipschitz for $\omega_h$-a.e. $\gamma$.

To obtain a curve that cannot be crossed from the left by any curves of the hypograph, we iterate this construction on small time intervals. For $\delta>0$ we let $t_k=k\delta$ and define an $S$-Lipschitz curve $x^\delta\colon [0,T]\to\R$ iteratively on $(t_0,t_1]$, $(t_1,t_2]$, etc., by setting
\begin{align*}
x^\delta(0)&=x_0,\\
x^\delta(t)&=\hat\gamma_{t_k,x^\delta(t_k)}(t)\quad\text{for }t\in (t_k,t_{k+1}]\cap [0,T],\quad k\delta\leq T.
\end{align*}
For all $t\in [0,T]$, the sequence $x^{1/2^{n}}(t)$ is monotone, and we set
\begin{align*}
x(t)=\lim_{n\to\infty} x^{1/2^{n}}(t)=\sup_{n>0} x^{1/2^{n}}(t).
\end{align*}
From the definitions of $x^\delta$ and $\hat\gamma_{t,x}$ we have that, for any $m\geq n>0$,
\begin{align*}
\omega_h\left(\left\lbrace \gamma\colon \exists t_1\in 2^{-n}\mathbb N,\, t_2>t_1,\, \gamma_x(t_1)<x^{1/2^m}(t_1),\,\gamma_x(t_2)>x^{1/2^m}(t_2)\right\rbrace\right)=0,
\end{align*}
and since $x(t)=\lim x^{1/2^m}(t)$ we deduce that
\begin{align*}
\omega_h\left(\left\lbrace \gamma\colon \exists t_1\in \bigcup_{n>0}(2^{-n}\mathbb N),\, t_2>t_1,\, \gamma_x(t_1)<x(t_1),\,\gamma_x(t_2)>x(t_2)\right\rbrace\right)=0.
\end{align*}
Property \eqref{left} follows because $x$ is Lipschitz, and so is $\gamma_x$ for any $\gamma\in\Gamma$.

Property \eqref{right} is a consequence of the fact, proven in \cite[Proposition~6]{marconi-rectif}, that curves of the epigraph cannot cross from the right curves of the hypograph, because $f$ is strictly convex. Specifically, exchanging the roles of $\omega_e$ and $\omega_h$ in \cite[Proposition~6]{marconi-rectif}, we have that $\omega_e$ is concentrated on a set $\Gamma_e$ such that, for any $\bar\gamma\in\Gamma_e$ and $t_1<t_2\in [0,T]$,
\begin{align}\label{eq:nocross}
\omega_h\left(\left\lbrace\gamma\colon \gamma_x(t_1)<\bar\gamma_x(t_1),\,\gamma_x(t_2)>\bar\gamma_x(t_2) \right\rbrace\right)=0.
\end{align}
Assuming that $\Gamma_e$ has non-empty intersection with the set in \eqref{right} and using the definition of $x(t)$, we would obtain a curve $\bar\gamma\in\Gamma_e$ and $t_1<t_2$ such that
\begin{align*}
\bar\gamma_x (t_1)> x^{1/2^n}(t_1) \quad\text{and}\quad \bar\gamma_x(t_2)< x^{1/2^n}(t_2),
\end{align*}
for some large enough $n$, and so there is $k\geq 0$, $t_k=k/2^n$, and $t\in (t_k,t_{k+1}]$ such that
\begin{align*}
\bar\gamma_x (t_k)\geq  x^{1/2^n}(t_k) \quad\text{and}\quad \bar\gamma_x(t) < x^{1/2^n}(t)=\hat\gamma_{t_k,x^{1/2^n}(t_k)}(t).
\end{align*}
By definition of $\hat\gamma_{t,x}$ this implies that
\begin{align*}
\omega_h\left(\left\lbrace \gamma\colon \gamma_x(t_k)<\bar\gamma_x(t_k),\,\gamma_x(t)>\bar\gamma_x(t)\right\rbrace\right)>0,
\end{align*}
thus contradicting \eqref{eq:nocross} and concluding the proof of \eqref{right}.
\end{proof}

The rest of Theorem~\ref{main}'s proof consists in showing that the curve $x(t)$ provided by Lemma~\ref{l:nocross} is a generalized characteristic. 
Thanks to the property \eqref{left} ensuring that curves of the hypograph typically  do not cross $x(t)$ from the left, in the flux formula \eqref{eq:Fgamma} for the flux $F_\gamma^-$ along a curve $\gamma$ across $x(t)$, there will be no contribution of the set $I_\gamma^+$ (times of crossings from left to right). Moreover, at typical times $t$ where the traces $u^\pm(t,x(t))$ agree, the contribution of the set $B_\gamma^-$ (times of bounces from the left) will be negligible. As a consequence, for monotone  entropies $\eta$ the flux across $x(t)$ will have a sign, providing a lower bound on $x'(t)$. The matching upper bound is then obtained similarly by using the property \eqref{right} that curves of the hypograph typically do not cross $x(t)$ from the right.

Before proceeding to the proof of Theorem~\ref{main}, let us be more specific about why the contribution of $B_\gamma^-$ will be negligible at points where the traces agree. This is due to the fact that entropy dissipation can be decomposed along the Lagrangian representation, and jumps in $\gamma_v$ create an amounts of entropy dissipation that is incompatible with the absence of jump $u^+(t,x(t))=u^-(t,x(t))$. We recall here the relevant result from \cite{marconi-structure}. We denote by $\nu$ the total entropy dissipation
\begin{align}\label{eq:nu}
\nu=\bigvee_{|\eta''|\leq 1 } |\mu_\eta|,
\end{align}
where $\bigvee$ stands for the lowest upper bound of a family of measures (as defined e.g. in \cite[Definition~1.68]{AFP}). As a consequence of \cite[Propositions~5.11 \& 5.12]{marconi-structure} there is a Lagrangian representation $\omega_h$ such that
\begin{align}\label{eq:nugamma}
\int_\Gamma |D\gamma_v|(\lbrace t\in A \colon \gamma_x(t)\in B\rbrace ) \, d\omega_h(\gamma) =\nu(A\times B),
\end{align}
for any borelian sets $A\subset [0,T]$, $B\subset\R$.

\begin{remark}
It is proved in \cite{marconi-rectif} that for Burgers equation the measure $\nu$ is actually equal to $|\mu_{\eta_0}|$ for $\eta_0(u)=u^2/2$, but we won't need it here.
\end{remark}

\begin{proof}[Proof of Theorem \ref{main}]
We assume without loss of generality that $u$ takes values in $[0,1]$ and show that the curve $x(t)$ from Lemma~\ref{l:nocross} is a generalized characteristic. As a consequence of \cite{delellis-otto-westdickenberg}, at almost every $t_0\in [0,T]$ such that the traces $u^\pm(t_0,x(t_0))$ are equal, we must have
\begin{align}\label{eq:nodissip}
\lim_{r\to 0}\frac{\nu(B_r(t_0,x_0))}{r}=0,
\end{align}
where $x_0=x(t_0)$ and $\nu$ is the total entropy dissipation \eqref{eq:nu}. Hence we fix $t_0$ a Lebesgue point of $x'$ such that $(t_0,x_0)$ is a Lebesgue point of $u^\pm$ satisfying \eqref{eq:nodissip}, and  prove that $x'(t_0)=f'(u_0)$, where $u_0=u^+(t_0,x(t_0))=u^-(t_0,x(t_0))\in (0,1)$.

Let $\eta$ be a nondecreasing $C^2$ entropy such that $\eta(0)=0$ and $q$ be the associated entropy flux with $q(0)=0$.
Given a non-negative test function $\Phi$ and applying the flux formula of Theorem~\ref{fluxformula} to the curve $x$, we obtain
\begin{align*}
& \int_0^T q(u^-(t,x(t))) - x'(t) \eta(u^-(t,x(t)))\, \Phi(t,x(t)) dt   \\
& = \int_\Gamma \sum_{t\in I_\gamma^+} \eta'(\gamma_v(t^-)) \Phi(t,x(t)) \, d\omega_h(\gamma)  -  \int_\Gamma \sum_{t\in I_\gamma^-} \eta'(\gamma_v(t^+)) \Phi(t,x(t))\, d\omega_h(\gamma) \\
 &\quad +   \int_\Gamma\sum_{t\in B_\gamma^-} (\eta'(\gamma_v(t^-))-\eta'(\gamma_v(t_+)))\, d\omega_h(\gamma).
\end{align*}
The first term in the right-hand side is zero thanks to the property \eqref{left} that typical curves of the hypograph don't cross $x(t)$ from the left. The second term is nonpositive because $\eta$ is non decreasing, so we deduce
\begin{align}\label{eq:estimflux}
& \int_0^T q(u^-(t,x(t))) - x'(t) \eta(u^-(t,x(t))) \Phi(t,x(t)) dt   \\
& \leq  \int_\Gamma\sum_{t\in B_\gamma^-} (\eta'(\gamma_v(t^-))-\eta'(\gamma_v(t_+)))\, d\omega_h(\gamma).\nonumber
\end{align}
Next we choose 
\begin{align*}
\Phi(t,x)=\chi_\delta(t)\varphi(x),\quad \chi_\delta(t)=\frac 1\delta \chi\left(\frac{t-t_0}{\delta}\right),
\end{align*}
where $\chi$ is a smooth cut-off function with $0\leq\chi(t)\leq \mathbf 1_{|t|\leq 1}$ and $\int\chi =1$, and $\varphi$ is any smooth compactly supported non-negative function such that $\varphi(x_0)=1$. Using the Lebesgue point properties of $t_0$ we may pass to the limit $\delta\to 0$ in the left-hand side of \eqref{eq:estimflux} and obtain
\begin{align*}
&q(u_0) - x'(t_0) \eta(u_0)  \\
&\leq \|\eta''\|_\infty  \limsup_{\delta\to 0}
\frac{1}{\delta} \int_\Gamma |D\gamma_v|(\lbrace t\in (t_0-\delta,t_0+\delta)\colon \gamma_x(t)=x(t)\rbrace) \, d\omega_h(\gamma).
\end{align*}
Since $x$ is $S$-Lipschitz and $x(t_0)=x_0$, using \eqref{eq:nugamma} to further estimate the right-hand side we infer
\begin{align*}
q(u_0) - x'(t_0) \eta(u_0)  \leq \|\eta''\|_\infty  \limsup_{\delta\to 0}
\frac{1}{\delta} \nu(B_{C \delta}(t_0,x_0)),
\end{align*}
where $C=\sqrt{1+S^2}$.
Recalling \eqref{eq:nodissip} we deduce
\begin{equation*}
q(u_0) - x'(t_0) \eta(u_0) \leq 0,
\end{equation*}
for any nondecreasing $C^2$ entropy $\eta$ with $\eta(0)=0$, and associated entropy flux $q$ with $q(0)=0$. Approximating by $C^2$ functions, this is valid for any nondecreasing $\eta$ with $\eta(0)=0$.  Choosing
\[
\eta (x) = (x-a) 1_{x \geq a}, q(x) = (f(x)-f(a)) 1_{x \geq a}
\]
for any $a\in [0,u_0)$, we deduce
\begin{align*}
x'(t_0)\geq \frac{f(u_0)-f(a)}{u_0-a},
\end{align*}
and letting $a\to u_0$ this implies $x'(t_0)\geq f'(u_0)$.

Using curves of the epigraph (see Remark~\ref{r:epiflux}) and the property \eqref{right} that typical curves of the epigraph cannot cross $x(t)$ from the right, we similarly obtain that
\begin{align*}
q(u_0) - x'(t_0) \eta(u_0) \geq 0
\end{align*}
for all nonincreasing entropy $\eta$ with $\eta(1)=0$, and associated entropy flux $q$ with $q(1)=0$, and applying this to
\[
\eta = (a-x) 1_{x \leq a}, q(x) = (f(a) - f(x)) 1_{x \leq a},
\]
for any $a\in (u_0,1]$ we deduce the opposite inequality $x'(t_0)\leq f'(u_0)$.
\end{proof}

\section{Proof of the Flux Formula}\label{s:flux}

In this section we prove Theorem~\ref{fluxformula}.  As in its statement, we fix $u$ a finite-entropy solution of \eqref{eq:scl} with Lagrangian representation $\omega_h$ of its hypograph, $\Sigma \subset (0,T)\times \R^d$ a Lipschitz hypersurface, and
an open set $U\subset (0,T)\times\R^d$ such that $U\setminus\Sigma$ has two connected components $\Sigma^\pm$.  We denote by $\nu$ the unit normal vector to $\Sigma$ pointing from $\Sigma^-$ to $\Sigma^+$ and by $u^\pm$ the traces of $u$ on the corresponding sides.

We start by establishing a first decomposition formula for the flux, where the flux along a curve $\gamma$ is not yet in the geometrically meaningful form of $F_\gamma^-$ in Theorem~\ref{fluxformula}.

\begin{lemma} \label{firstformula}
For any entropy-entropy flux pair $(\eta,q)$ with $\eta(0)=0$, $q(0)=0$ we have
\begin{equation}\label{eq:firstformula}
\int_\Sigma (\eta(u^-) , q(u^-))\cdot \nu\,\Phi(t,x)\, d\mathcal H^d(t,x) = \int_\Gamma \langle \widetilde{F}_\gamma,\eta\otimes\Phi\rangle d\omega_h(\gamma)\qquad\forall \Phi\in C_c^\infty(U),
\end{equation}
where
\begin{align*}
\langle \widetilde{F}_\gamma ,\eta \otimes \Phi \rangle & =  - \int_0^T \mathbf 1_{(t,\gamma_x(t))\in \overline{\Sigma^+}}\Phi(t,\gamma_x(t)) D(\eta'\circ \gamma_v)(dt) \nonumber \\
& \quad -  \int_0^T \mathbf 1_{(t,\gamma_x(t))\in \overline{\Sigma^+}} \eta'(\gamma_v(t))(\partial_t\Phi+f'(\gamma_v(t))\cdot\nabla_x\Phi)(t,\gamma_x(t)) dt.
\end{align*}
\end{lemma}

\begin{proof}
For small enough $\delta>0$ we may find a $C/\delta$-Lipschitz function $G_\delta\colon U\to [0,1]$, with $C>0$ independent of $\delta$, such that
\begin{align*}
&G_\delta=1\text{ on }\Sigma^+,\quad G_\delta=0 \text{ on }\left\{(t,x) \in \Sigma^- : \dist((t,x), \Sigma^+)\ge \delta \right\}\\
\text{and }&\nabla G_\delta \to \nu \otimes \mathcal H^d_{\lfloor \Sigma\cap U}\text{ as }\delta\to 0.
\end{align*}
Then we have
\begin{align*}
&\int_\Sigma (\eta(u^-) , q(u^-))\cdot \nu\,\Phi(t,x)\,
 d\mathcal H^d(t,x)\\
&=\lim_{\delta\to 0}\int (\eta(u),q(u))\cdot \nabla G_\delta(t,x)  \,\Phi(t,x)\, dtdx
\end{align*}
Since $\eta(0)=0$ and $q(0)=0$ we may write $\eta(u)=\int \mathbf 1_{v<u}\eta'(v)\, dv$ and $q(u)=\int \mathbf 1_{v<u}\eta'(v)f'(v)\, dv$. Combining this with  property \eqref{eq:etomega} of the Lagrangian representation, $(e_t)\sharp\omega_h=\mathbf 1_{v<u(t,x)}\, dx\, dv$, we find
\begin{align*}
&\int (\eta(u),q(u))\cdot \nabla G_\delta(t,x)  \,\Phi(t,x)\, dtdx\\
&=
\int_0^T (1,f'(\gamma_v(t)))\cdot \nabla G_\delta(t,\gamma_x(t)) \, \eta'(\gamma_v(t))\Phi(t,\gamma_x(t))\, dt\, d\omega_h(\gamma)\\
&= \int_{\Gamma} \int_0^T \frac{d}{dt}\left[G_\delta(t,\gamma_x(t))\right] \eta'(\gamma_v(t))\Phi(t,\gamma_x(t))\, dt\, d\omega_h(\gamma)\\
&=- \int_{\Gamma}\Bigg[ \int_0^T G_\delta(t,\gamma_x(t))\Phi(t,\gamma_x(t))\, D(\eta'\circ \gamma_v)(dt) \\
&\hspace{3em} +\int_0^T G_\delta(t,\gamma_x(t))\eta'(\gamma_v(t))(\partial_t\Phi+f'(\gamma_v(t))\cdot\nabla_x\Phi)(t,\gamma_x(t))\, dt \Bigg]\, d\omega_h(\gamma).
\end{align*}
For the second equality we used the fact that $\omega_h$ is concentrated on curves satisfying the characteristic equation \eqref{eq:characlag}.
Note that $G_\delta (t,x)\to \mathbf 1_{(t,x)\in\overline{\Sigma^+}}$ for all $(t,x)\in U$ as $\delta\to 0$, and that the Lagrangian representation satisfies
$\int_\Gamma (1+\mathrm{TotVar}\:\gamma_v\, d\omega_h(\gamma)<\infty$ \eqref{eq:totvarlag}. Hence by dominated convergence we deduce 
\begin{align*}
\lim_{\delta\to 0}\int (\eta(u),q(u))\cdot \nabla G_\delta(t,x)  \,\Phi(t,x)\, dtdx
= \int_\Gamma \langle \widetilde{F}_\gamma,\eta\otimes\Phi\rangle d\omega_h(\gamma),
\end{align*} 
which concludes the proof of \eqref{eq:firstformula}.
\end{proof}

Next we check that $\widetilde F_\gamma =F_\gamma^-$ provided $\gamma$ intersects $\Sigma$ at most a finite number of times.

\begin{lemma} \label{intermediateformula}
If $\gamma\in\Gamma$ intersects $\Sigma$ at most a finite number of times, that is,
\begin{align*}
N_\Sigma(\gamma)=\mathrm{card}\: \left\lbrace t\in (0,T)\colon (t,\gamma_x(t))\in\Sigma\right\rbrace <\infty,
\end{align*}
 then
\begin{align*}
\langle \widetilde{F}_\gamma, \eta \otimes \Phi \rangle = \langle F_\gamma, \eta \otimes \Phi \rangle \qquad\forall \eta\in C^1(\R),\Phi\in C_c^\infty(U),
\end{align*}
where $\widetilde F_\gamma$ and $F_\gamma^-$ are defined in Lemma~\ref{firstformula} and Theorem~\ref{fluxformula}.
\end{lemma}
\begin{proof}
Setting
\begin{align}\label{eq:thetapsi}
\theta(t)=\mathbf 1_{(t,\gamma_x(t))\in \overline{\Sigma^+}},\qquad 
\psi(t)=\eta'(\gamma_v(t))\Phi(t,\gamma_x(t)),
\end{align}
we rewrite $\widetilde F_\gamma$ as
\begin{align}
\label{eq:Fthetapsi}
\langle \tilde{F}_\gamma, \eta \otimes \Phi \rangle   = -\int_0^T \theta(t) D\psi(dt). 
\end{align}
Since $\gamma$ has a finite number of intersections with $\Sigma$, we know that the $\lbrace 0,1\rbrace$-valued function $\theta$ is  $BV$. Its jump set can be decomposed as
\begin{align*}
J_\theta =I_\gamma^+ \cup I_\gamma^-,
\end{align*}
where $I_\gamma^\pm$ are as in Theorem~\ref{fluxformula} the sets of intersection times where $\gamma$ crosses $\Sigma$ from $\Sigma^\mp$ to $\Sigma^\pm$. They correspond to positive and negative jumps of $\theta$. Note that $\theta(t)=1$ for all $t\in J_\theta$. Moreover, the set $B_\gamma^-$ of intersection times where $\gamma$ bounces on $\Sigma$ from $\Sigma^-$ corresponds to the non-jump points of $\theta$ at which its pointwise value is different from its left and right limits: for $t\in B_\gamma^-$ we have $\theta(t^+)= \theta(t^-)=0$ but $\theta(t)=1$. 

The function $\psi$ is also $BV$, and its jump set is included in $J_\gamma$, the jump set of $\gamma$. We use the product rule in $BV$ for $D(\theta\psi)$ in order to rewrite \eqref{eq:Fthetapsi}. 
As in \cite{AFP} we denote by $\tilde D\theta$ the sum of the absolutely continuous and Cantor parts of the measure $D\theta$.
Since $\widetilde D\theta=0$ and $\psi$ vanishes at the boundary of $(0,T)$, we obtain
\begin{align*}
&\langle \tilde{F}_\gamma, \eta \otimes \Phi \rangle 
\\
&=
 -\int_0^T \theta(t) \tilde D\psi(dt) -\sum_{t\in J_\gamma} \theta(t) (\psi(t^+)-\psi(t^-)) \nonumber \\
&=\sum_{t\in J_\theta\cup J_\gamma}   (\theta(t^+)\psi(t^+)-\theta(t^-)\psi(t^-)) -\sum_{t\in J_\gamma} \theta(t) (\psi(t^+)-\psi(t^-))\\
&=\sum_{t\in J_\theta\setminus J_\gamma} \psi(t)   (\theta(t^+)-\theta(t^-)) 
+ \sum_{t\in J_\gamma\cap J_\theta}   (\theta(t^+)\psi(t^+)-\theta(t^-)\psi(t^-)) \\
& \quad -\sum_{t\in J_\gamma\cap J_\theta} (\psi(t^+)-\psi(t^-))
-\sum_{t\in J_\gamma\cap B_\gamma^-} (\psi(t^+)-\psi(t^-))\\
&= \sum_{t\in  I_\gamma^+ }   \psi(t^-) - \sum_{t\in I_\gamma^-}   \psi(t^+)
 +\sum_{t\in B_\gamma^-} (\psi(t^-)-\psi(t^+)),
\end{align*}
and recalling the definitions of $\theta,\psi$ \eqref{eq:thetapsi} this corresponds exactly to $\langle F_\gamma, \eta \otimes \Phi \rangle$.
\end{proof}

Now Theorem~\ref{fluxformula} follows directly from Lemmas~\ref{firstformula} and \ref{intermediateformula}, provided we show that $\omega_h$-a.e. $\gamma$ intersects $\Sigma$ a finite number of times:

\begin{proposition}\label{p:finiteinter}
Let $\Sigma\subset (0,T)\times \R^d$ be a Lipschitz hypersurface. Then
\begin{align*}
\omega_h\left(\left\lbrace \gamma\in\Gamma\colon N_\Sigma(\gamma)=\infty\right\rbrace\right)=0.
\end{align*}
\end{proposition}

Proposition~\ref{p:finiteinter} follows from the following:

\begin{lemma}\label{l:graphfinite}
Let $h:[0,T] \times \R^{d-1}\to\R$ be Lipschitz and 
\[
\Sigma = \{ (t , \hat x , h(t , \hat x))\colon  (t,\hat x)\in [0,T]\times \R^{d-1} \},
\]
then $N_\Sigma(\gamma)<\infty$ for $\omega_h$-a.e. $\gamma\in\Gamma$.
\end{lemma}

\begin{proof}[Proof of Proposition~\ref{p:finiteinter}]
First, we know that $\Sigma$ is a locally finite union of Lipschitz graphs of the form $x_i=h(\hat x_i, t)$ for some $i\in\lbrace 1,\ldots,d\rbrace$, where $\hat x_i=(x_1,\ldots,x_{i-1},x_{i+1},\ldots,x_d)$ or of the form $t = g(x)$. Moreover for the graphs over $x$ we can assume that $|\nabla g|<1/S$, where $S$ is the maximal speed for curves $\gamma_x$. Therefore, $\omega_h$-a.e $\gamma$ can have at most one intersection with such a graph. For the rest of the graphs, we apply Lemma~\ref{l:graphfinite} to each graph, and we conclude that $N_\Sigma(\gamma)$ is finite for $\omega_h-$ a.e. $\gamma \in \Gamma$.
\end{proof}

We will now prove Lemma~\ref{l:graphfinite}. The strategy is to rule out non-transverse intersections, but this makes sense only at points of the hypersurface which admit a tangent space. Hence we need a preliminary result allowing us to leave out the non-differentiable points:

\begin{lemma} \label{l:nullinter}
Let $E \subset [0,T] \times \R^d$ be such that $\mathcal{H}^d (E) =0.$ Then
\[
\omega_h(\{ \gamma \in \Gamma : \exists t \in [0,T] \text{ such that } (t,\gamma_x(t)) \in E \}) = 0.
\]
\end{lemma}

\begin{proof}[Proof of Lemma~\ref{l:nullinter}]
Since $\mathcal{H}^d(E) = 0,$ for any $\e> 0$ there is a sequence of balls $B((t_i,x_i),r_i) $ such that
\[
E \subset \cup_{i=1}^\infty B((t_i,x_i),r_i),
\qquad 
\sum_{i=1}^\infty r_i^d < \e.
\]
For $\omega_h$-a.e. $\gamma$, since $\gamma_x$ is $S$-Lipschitz we have the implication
\begin{align*}
\Big(\exists t\in [0,T],\; (t,\gamma_x(t))\in B((t_i,x_i),r_i) \Big) \quad\Longrightarrow \quad \gamma_x(t_i)\in B(x_i,(1+S)r_i),
\end{align*}
and using the property \eqref{eq:etomega} that $(e_t)\sharp\omega_h =\mathbf 1_{v<u(t,x)}\, dx dv$ we deduce
\begin{align*}
&\omega_h \left(\left\lbrace \gamma\colon\exists t\in [0,T],\,(t,\gamma_x(t))\in B((t_i,x_i),r_i)) \right\rbrace \right)\\
&\leq ( e_{t_i})\sharp \omega_h \left( B(x_i,(1+S)r_i)\times [0,1]\right) \leq C r_i^d, 
\end{align*}
for some constant $C>0$. Summing over $i$ this implies
\[
\omega_h (\{ \gamma \colon \exists t \in [0,T] ,\; (t,\gamma_x(t)) \in E \}) \leq C \sum_{i=1}^\infty r_i^d < C \e,
\]
and letting $\e\to 0$ concludes the proof of Lemma~\ref{l:nullinter}.
\end{proof}

With Lemma~\ref{l:nullinter} at hand we now prove Lemma~\ref{l:graphfinite}.

\begin{proof}[Proof of Lemma \ref{l:graphfinite}]
Let $R,\varepsilon>0,$ and $\delta > 0.$ Since $\Sigma$ is Lipschitz, it has a tangent plane at almost every point. We let $\bar{\Sigma}$ be the set of points on $\Sigma$ that have a tangent plane. Then by Lemma~\ref{l:nullinter}, we see that
\[
\omega_h(\{ \gamma \in \Gamma : \exists t \in [0,T], \mbox{ s.t. } (t,\gamma_x(t)) \in \Sigma \setminus \bar{\Sigma} \}) = 0,
\]
so we only need to show that $N_{\bar\Sigma}(\gamma)<\infty$ for $\omega_h$-a.e. $\gamma$. We obtain this by proving that for $\omega_h$-a.e. $\gamma$, all intersections with $\bar\Sigma$ must be transverse, i.e.
\begin{align}\label{eq:transverse}
&\omega_h\left(X^-\right)=\omega_h\left(X^+\right)=0,\quad\text{where }\\
&X^-=\left\lbrace
\gamma\colon \exists t_0\in (0,T],\; (t_0,\gamma_x(t_0))\in \bar\Sigma,\;(1, \gamma_x'(t_0^-))\in T_{(t_0,\gamma_x(t_0))}\Sigma \nonumber
\right\rbrace, \\
&X^+=\left\lbrace
\gamma\colon \exists t_0\in [0,T),\; (t_0,\gamma_x(t_0))\in \bar\Sigma,\;(1, \gamma_x'(t_0^+))\in T_{(t_0,\gamma_x(t_0))}\Sigma \nonumber
\right\rbrace.
\end{align}
Granted \eqref{eq:transverse} we directly deduce that for $\omega_h$-a.e. $\gamma$ the set of intersection times with $\bar\Sigma$ is discrete, and therefore finite. 

\medskip

It remains to prove \eqref{eq:transverse}. We will prove only $\omega_h\left(X^-\right)=0$, the argument for $X^+$ being analogous.First, we remark that the set $X^-$ satisfies
\begin{align}\label{eq:XRedelta}
&X^-\subset \bigcup_{R>0}\bigcap_{\e>0}\bigcup_{\delta>0} X_{R,\e}^\delta,\quad\text{where }\\
&
X_{R,\varepsilon}^\delta =\Big\lbrace \gamma\in \Gamma\colon  |\gamma_x(0)|\leq R\text{ and } \exists t_0\in [0,T],\; (t_0,\gamma_x(t_0))\in \bar{\Sigma},\nonumber\\
&\hspace{7em} \mathrm{dist}((t,\gamma_x(t)), \Sigma \cap (\{t\} \times B_{R+ST})) \leq \e (t_0 - t)\;\forall t\in (t_0-\delta,t_0)\Big\rbrace.\nonumber
\end{align}

This follows directly from the definitions of $\gamma_x'$ and $T\Sigma$.
Let indeed $\gamma\in X^-$. There obviously is an $R>0$ such that $|\gamma_x(0)|<R$, and then $\gamma_x(t)\in B_{R+ST}$ for all $t$ since $|\gamma_x'|\leq S$. 
Suppose that $\gamma$ intersects $\bar{\Sigma}$ tangentially at time $t_0$, that is, $\gamma_x(t_0)=x_0=(\hat x_0,h(t_0,\hat x_0))$, where $\hat x_0\in\R^{d-1}$ denotes the first $(d-1)$ components of $x_0$, and
\begin{align*}
\gamma_x'(t_0^-)=(\hat y_0,\nabla h(t_0,\hat x_0)\cdot(1,\hat y_0))\quad\text{for some }\hat y_0\in\R^{d-1}.
\end{align*}
Let $\e>0$. By definition of $\gamma_x'(t_0^-)$ and $\nabla h(t_0,x_0)$, there exists $\delta>0$ such that, for all $s\in (-\delta,0]$,
 \begin{align*}
  &|\gamma_x(t_0+s)-x_0-(s\hat y_0,\nabla h(t_0,\hat x_0)\cdot(s,s\hat y_0))|\leq\frac \e 2 |s|,\\
 &|h(t_0+s,\hat x_0 + s\hat y_0)-h(t_0,\hat x_0)-\nabla h(t_0,\hat x_0)\cdot (s,s\hat y_0)|
 \leq
 \frac \e 2 |s|,
 \end{align*}
 which implies
 \begin{align*}
 |(t_0+s,\gamma_x(t_0+s)) - (t_0+s,\hat x_0 +s \hat y_0,h(\hat x_0) + h(t_0+s,\hat x_0+s\hat y_0)|\leq \e |s|,
 \end{align*}
 and therefore 
 $\mathrm{dist}((t,\gamma_x(t)), \Sigma \cap (\{t\} \times B_{R'})) \leq \e (t_0 - t)$ for all $t\in (t_0-\delta,t_0]$, proving \eqref{eq:XRedelta}.

\medskip 

Next we claim that the sets $X_{R,\e}^\delta$ defined in \eqref{eq:XRedelta} satisfy
\begin{align}\label{eq:estimXRedelta}
\limsup_{\delta\to 0}\omega_h(X_{R,\varepsilon}^\delta) \leq C\e,
\end{align}
for some constant $C>0$ depending only on $S$, $R$, $T$ and $\Sigma$. Since the union $\bigcup_\delta X_{R,\e}^\delta$ is nondecreasing, this implies via \eqref{eq:XRedelta} that $\omega_h(X^-)=0$ and concludes the proof of Lemma~\ref{l:graphfinite}.

Recall the definition  
\begin{align*}
X_{R,\varepsilon}^\delta =\Big\lbrace \gamma\in \Gamma\colon & |\gamma_x(0)|\leq R\text{ and } \exists t_0\in [0,T],\; (t_0,\gamma_x(t_0))\in \bar{\Sigma},\\
& \mathrm{dist}((t,\gamma_x(t)), \Sigma \cap (\{t\} \times B_{R+ST})) \leq \e (t_0 - t)\;\forall t\in (t_0-\delta,t_0)\Big\rbrace.
\end{align*}
and let $N:=\lfloor T/\delta\rfloor$. For any $\gamma\in X_{R,\e}^\delta$, there exists $k\in \lbrace 1,\ldots, N\rbrace$ such that
\[
\mathrm{dist}( (k/N,\gamma_x(k/N)) , \Sigma \cap (\{ k/N \} \times B_{R+ST}) ) < \e \delta,
\]
and therefore
\[
X_{R,\e}^\delta \subseteq \bigcup_{k=1}^N \left\{ \gamma \in \Gamma: \left( \frac{k}{N} , \gamma_x \left( \frac{k}{N} \right) \right)  \in \Sigma \cap (\{ k/N \} \times B_{R+ST}) + B(0,\e \delta)\right\}.
\]
Using the Lagrangian property \eqref{eq:etomega}, we know that $e_{k/N}\sharp \omega_h \leq \mathcal L^{d+1},$ and thus
\begin{align*}
\omega_h(X_{R,\varepsilon}^\delta) 
& \leq \sum_{k=1}^N \mathcal{L}^{d+1} \left(\left(\Sigma \cap \left( \left\{ \frac{k}{N} \right\} \times B_{R+ST} \right) + B(0,\e \delta)\right) \times [0,1] \right) \\
& = \sum_{k=1}^N \mathcal{L}^{d} \left(\Sigma \cap \left( \left\{ \frac{k}{N} \right\} \times B_{R+ST} \right) + B(0,\e \delta) \right). 
\end{align*}
Letting $L$ denote the Lipschitz constant of $h$, we have
\begin{align*}
&\Sigma \cap \left( \left\{ \frac{k}{N} \right\} \times B_{R+ST} \right) + B(0,\e \delta) \\
&\subset \left\lbrace x=(\hat x, x_d)\in B_{R+ST}\colon |x_d-h(k/N,\hat x)|\leq (L+1)\e\delta\right\rbrace,
\end{align*}
and we deduce that
\begin{align*}
\mathcal{L}^{d} \left(\Sigma \cap \left( \left\{ \frac{k}{N} \right\} \times B_{R+ST} \right) + B(0,\e \delta) \right)
\leq c (L+1) \e \delta \,(R+ST)^{d-1},
\end{align*}
for some absolute constant $c>0$.
Therefore we have
\[
\omega_h (X_{R , \e}^ \delta) \leq c(L+1)\e \,(R+ST)^{d-1} \, N\delta
\leq c(L+1)\, (R+ST)^{d-1} T \,\e,
\]
which implies \eqref{eq:estimXRedelta}.
\end{proof}

\section*{Acknoledgements}
X.L. is partially supported by ANR project ANR-18-CE40-0023 and COOPINTER project IEA-297303.
E.M. is supported by the SNF Grant 182565.

\bibliographystyle{acm}

\end{document}